\documentclass[a4paper,11pt]{amsart}
\usepackage{amssymb,amsmath,amscd,tikz,url}
\usepackage[all]{xy}
\usepackage{hyperref}

\newtheorem{thm}{Theorem}[section]

\newtheorem{cor}[thm]{Corollary}
\newtheorem{lem}[thm]{Lemma}

\newtheorem{defn}[thm]{Definition}

\newtheorem{Condition}[thm]{Condition}

\newcommand{\Zset}{\mathbb{Z}}

\newcommand{\Cset}{\mathbb{C}}

\def\top{{\rm top}}

\newcommand{\ab}{{\rm ab}}

\def\fo{{\mathfrak o}}

\def\fU{{\mathfrak U}}

\def\PPhi{{\Phi}}

\def\cB{{\mathcal B}}

\def\Z{{\rm Z}}

\def\im{{\rm im\,}}

\def\H{{\rm H}}
\def\HP{{\rm HP}}

\def\SL{{\rm SL}}

\def\SO{{\rm SO}}

\def\GL{{\rm GL}}

\def\der{{\rm der}}

\def\Hom{{\rm Hom}}

\def\top{{\rm top}}

\def\unr{{\rm unr}}

\def\en{{\rm en}}
\def\aff{{\rm aff}}

\def\bW{{\mathbf W}}

\def\der{{\rm der}}

\def\Irr{{\mathbf {Irr}}}

\def\aff{{\rm a}}

\def\cR{{\mathcal R}}
\def\cG{{\mathcal G}}

\def\cH{{\mathcal H}}

\def\cO{{\mathcal O}}

\def\cT{{\mathcal T}}

\def\cW{{\mathcal W}}

\def\fs{{\mathfrak s}}

\def\fii{{\mathfrak i}}

\def\aff{{\rm aff}}

\def\q{{/\!/}}

\def\ab{{\rm ab}}
\def\der{{\rm der}}

\def\PPhi_F{{\rm Frob}}

\def\cpt{{\rm cpt}}

\newcommand{\Q}{\mathbb Q}
\newcommand{\R}{\mathbb R}
\newcommand{\C}{\mathbb C}
\newcommand{\matje}[4]{\left(\begin{smallmatrix} #1 & #2 \\ 
#3 & #4 \end{smallmatrix}\right)}
\def\lexp#1#2{{\kern\scriptspace\vphantom{#2}^{#1}\kern-\scriptspace#2}}

\begin{document}
\title[Geometric structure]{Geometric structure for the principal series 
of a split reductive $p$-adic group with connected centre}

\author[A.-M. Aubert]{Anne-Marie Aubert}
\address{I.M.J.-PRG, U.M.R. 7586 du C.N.R.S., U.P.M.C., Paris, France}
\email{anne-marie.aubert@imj-prg.fr}
\author[P. Baum]{Paul Baum}
\address{Mathematics Department, Pennsylvania State University,  University Park, PA 16802, USA}
\email{baum@math.psu.edu}
\thanks{The second author was partially supported by NSF grant DMS-0701184}
\author[R. Plymen]{Roger Plymen}
\address{School of Mathematics, Southampton University, Southampton SO17 1BJ,  England 
\emph{and} School of Mathematics, Manchester University,
Manchester M13 9PL, England}
\email{r.j.plymen@soton.ac.uk \quad plymen@manchester.ac.uk}
\author[M. Solleveld]{Maarten Solleveld}
\address{Radboud Universiteit Nijmegen, Heyendaalseweg 135, 6525AJ Nijmegen, the Netherlands}
\email{m.solleveld@science.ru.nl}

\date{\today}
\maketitle

\begin{abstract}  Let $\cG$ be a split reductive $p$-adic group with connected centre.   
We show that each Bernstein block in the principal series of $\cG$ admits a definite geometric 
structure, namely that of an extended quotient.   For the Iwahori-spherical block, this extended 
quotient has the form $T\q W$ where $T$ is a maximal torus in the Langlands dual group of $\cG$ 
and  $W$ is the Weyl group of $\cG$.   
\end{abstract}

\tableofcontents

\section{Introduction}

Let $\cG$ be a split reductive $p$-adic group with connected centre,  \
and let $G = \cG^\vee$ denote the Langlands dual group.   Then $G$ is a complex reductive group.
Let $T$ be a maximal torus in $G$ and let $W$ be the common Weyl group of $\cG$ and $G$.   
We can form the quotient variety
\[T/W
\]
and, familiar from noncommutative geometry \cite[p.77]{K}, the \emph{noncommutative quotient algebra}
\[
\cO(T) \rtimes W .
\]
Within periodic cyclic homology (a noncommutative version of de Rham theory) 
there is a canonical isomorphism
\[
\HP_*(\cO(T) \rtimes W) \simeq \H^*(T\q W ; \C)
\]
where 
\[
T\q W
\]
 denotes the \emph{extended quotient} of $T$ by $W$, see \S \ref{sec:extquot}.   
In this sense, the extended quotient $T\q W$, a complex algebraic variety,  
is a more concrete version of the noncommutative quotient algebra $\cO(T) \rtimes W$.   

Returning to the $p$-adic group $\cG$, let $\Irr(\cG)^{\fii}$ denote the subset of the smooth dual 
$\Irr(\cG)$ comprising all the irreducible smooth Iwahori-spherical representations of $\cG$.
We prove in this article that there is a continuous bijective map, satisfying several 
constraints, as follows:
\[
T\q W \simeq \Irr(\cG)^{\fii}.
\]
We note that there is nothing in the classical representation theory of $\cG$ 
to indicate that $\Irr(\cG)^{\fii}$ admits such a geometric structure.
Nevertheless, such a structure was conjectured by the present authors in \cite{ABPS1}, 
and so this article is a confirmation of that conjecture, for the single point $\fii$ 
in the Bernstein spectrum  of $\cG$.      We prove, more generally, that,
subject to constraints itemized in \cite{ABPS1}, and subject to the Condition \ref{CC} 
on the residual characteristic, there is a continuous bijective map
\[
T^\fs \q W^\fs \simeq \Irr(\cG)^\fs
\]
for each point $\fs$ in the Bernstein spectrum for the principal series of $\cG$, 
see Theorem \ref{split}. Here, $T^\fs$ and $W^\fs$ are the complex torus 
and the finite group attached to $\fs$.   
This, too, is a confirmation of the geometric conjecture in \cite{ABPS1}.

Let $\bW_F$ denote the Weil group of $F$.   
A Langlands parameter $\Phi$ for the principal series of 
$\cG$ 
should have
$\Phi (\mathbf{W}_F)$ contained in a maximal torus of $G$. In particular, it should suffice to consider 
parameters $\Phi$ such that
$\Phi \big|_{\mathbf{W}_F}$ factors through $\mathbf{W}_F^{\ab} \cong F^\times$,  that is, such that $\Phi$ factors as follows:
\[
\Phi \colon\mathbf{W}_F\times \SL_2(\C) \to F^{\times} \times \SL_2(\C) \to G.
\]
Such a parameter is \emph{enhanced} in the following way.   Let $\rho$ be an irreducible representation 
of the component group of the centralizer of the image of $\Phi$:
\[
\rho \in \Irr \, \pi_0 ( Z_G(\im \, \Phi)) .
\]
The pair $(\Phi, \rho)$ will be called an \emph{enhanced Langlands parameter}.   
   
We rely on Reeder's classification of the constituents of a given principal series 
representation of $\cG$, see \cite[Theorem 1, p.101-102]{Reed}.    
Reeder's theorem amounts to a local Langlands correspondence for the principal series 
of $\cG$.    Reeder uses only enhanced Langlands parameters with a particular geometric origin, 
namely those which occur in the homology of a certain variety of Borel subgroups of $G$.
This condition is essential, see, for example, the discussion, in \cite{ABP2},  
of the  Iwahori-spherical representations of the exceptional group $G_2$.

In Theorem \ref{compareParameters} we show how to replace the enhanced Langlands 
parameters of this kind, namely those of geometric origin,
by the \emph{affine Springer parameters} defined in \S \ref{par:affSpringer}.   
These affine Springer 
parameters are defined in terms of data attached to the complex reductive group $G$ -- 
in this sense, the affine Springer parameters are independent of the cardinality $q$ of the residue 
field of $F$.   The scene is now set for us to prove the first theorem of geometric structure, 
namely Theorem \ref{thm:bijection}, from which our main structure theorem, 
Theorem \ref{split} follows.

We also relate our basic structure theorem with $L$-packets in the principal series of $\cG$, 
see Theorem \ref{Lpackets}.  

An earlier, less precise version of our conjecture was formulated in \cite{ABP1}. 
That version was proven in \cite{Sol} for Bernstein components which are described nicely 
by affine Hecke algebras. These include the principal series of split groups 
(with possibly disconnected centre), symplectic 
and orthogonal groups and also inner forms of $\GL_n$. \\ 

\textbf{Acknowledgements.}
Thanks to Mark Reeder for drawing our attention to the article of Kato \cite{Kat}.   
We thank Joseph Bernstein, David Kazhdan, George Lusztig, and David Vogan for 
enlightening comments and discussions.

\section{Extended quotient}
\label{sec:extquot}
Let $\Gamma$ be a finite group acting on a complex affine variety $X$ 
by automorphisms,
\[ 
\Gamma \times X \to X.
\]
The quotient variety $X/\Gamma$ is obtained by collapsing each orbit to a point. 

For $x\in X ,\; \Gamma_x$ denotes the stabilizer group of $x$:
\[
\Gamma_x = \{\gamma\in \Gamma : \gamma x = x\}.
\]
Let $c(\Gamma_x)$ denote the set of conjugacy classes of  $\Gamma_x$. The extended quotient 
is obtained from $X / \Gamma$ by replacing the orbit of $x$ by $c(\Gamma_x)$. 
This is done as follows:\\

\noindent Set $\widetilde{X} = \{(\gamma, x) \in \Gamma \times X : \gamma x = x\}$.  
It is an affine variety and a subvariety of $\Gamma \times X$. 
The group $\Gamma$ acts on $\widetilde{X}$:
\begin{align*}
& \Gamma \times \widetilde{X} \to \widetilde{X}\\
& \alpha(\gamma, x) = (\alpha\gamma \alpha^{-1}, \alpha x), \quad\quad \alpha \in \Gamma,
\quad (\gamma, x) \in \widetilde{X}.
\end{align*}

\noindent The extended quotient, denoted $ X/\!/\Gamma $,  is $\widetilde{X}/\Gamma$. 
Thus the extended quotient  $ X/\!/\Gamma $ is the usual quotient for the action of 
$\Gamma$ on $\widetilde{X}$. 
The projection 
\[
\widetilde{X} \to X ,\;  (\gamma, x) \mapsto x
\]
is $\Gamma$-equivariant 
and so passes to quotient spaces to give a morphism of affine varieties
\[
\rho\colon  X/\!/\Gamma \to X/\Gamma.
\]   
This map will be referred to as the projection of the extended quotient 
onto the ordinary quotient. The inclusion 
\begin{align*}
& X \hookrightarrow \widetilde{X}\\
& x \mapsto (e,x)\qquad e=\text{identity element of }\Gamma
\end{align*}
is $\Gamma$-equivariant and so passes to quotient spaces to give an inclusion of affine 
varieties $X/\Gamma\hookrightarrow X/\!/\Gamma$. 

This article will be dominated by extended quotients of the form
$T\q W$ or, more generally, extended quotients of the form $T^\fs \q W^\fs$.

\section{The group $W^\fs$ as a Weyl group}
\label{sec:Lp}
Let $\cG$ be a connected reductive $p$-adic group over $F$, which is $F$-split and
has connected centre. Let $\mathcal T$ be a $F$-split maximal torus in $\cG$. Let
$G$, $T$ denote the Langlands dual groups of $\mathcal{G}$, $\mathcal{T}$. 
The principal series consists 
of all $\mathcal G$-representations that are obtained with parabolic induction
from characters of $\mathcal T$. 
We will suppose that the residual characteristic $p$ of $F$ satisfies the hypothesis 
in \cite[p.~379]{Roc}, for all reductive subgroups $H \subset G$ containing $T$:

\begin{Condition}\label{CC} 
If the root system $R (H,T)$ is irreducible, 
then the restriction on the residual characteristic $p$ of $F$ is as follows:
\begin{itemize}
\item for type $A_n  \quad p > n+1$
\item for types $B_n, C_n, D_n \quad p \neq 2$
\item for type $F_4 \quad p \neq 2,3$
\item for types $G_2, E_6 \quad p \neq 2,3,5$
\item for types $E_7, E_8 \quad p \neq 2,3,5,7.$
\end{itemize}
If $R (H,T)$ is reducible, one excludes primes attached to each of its 
irreducible factors.
\end{Condition}
Since $R (H,T)$ is a subset of $R (G,T) \cong R (\cG,\cT)^\vee$, 
these conditions are fulfilled when they hold for $R (\cG,\cT)$.

We denote the collection of all Bernstein components of $\cG$ of the form
$\fs=[\cT,\chi ]_\cG$ by $\mathfrak B (\cG,\cT)$ and call these the Bernstein
components in the principal series. The union 
\[
\Irr (\cG,\cT) := \bigcup_{\fs \in \mathfrak B (\cG,\cT)} \Irr (\cG )^\fs
\]
is by definition the set of all irreducible subquotients of principal series
representations of $\mathcal G$.

Choose a uniformizer $\varpi_F \in F$. There is a bijection $t \mapsto \nu$ between 
points in $T$ and unramified characters of $\mathcal{T}$, determined by the relation 
\[
\nu (\lambda(\varpi_F)) = \lambda (t) 
\]
where $\lambda \in X_* (\mathcal{T}) = X^* (T)$.   
The space $\Irr (\cT )^{[\cT ,\chi]_\cT}$ is in bijection with $T$ via 
$t \mapsto \nu \mapsto \chi \otimes \nu$. Hence Bernstein's torus $T^\fs$ is isomorphic
to $T$. However, because the isomorphism is not canonical and the action of the group
$W^\fs$ depends on it, we prefer to denote it $T^\fs$.

The uniformizer $\varpi_F$
gives rise to a group isomorphism $\fo_F^\times \times \mathbb Z \to F^\times$,
which sends $1 \in \mathbb Z$ to $\varpi_F$.
Let $\cT_0$ denote the maximal compact subgroup of $\cT$. As the latter is $F$-split,
\begin{equation}\label{eq:cT0}
\cT \cong F^\times \otimes_{\mathbb Z} X_* (\cT) \cong (\fo_F^\times \times \mathbb Z)
\otimes_{\mathbb Z} X_* (\cT) = \cT_0 \times X_* (\cT) .
\end{equation}
Because $\cW^G = W (\cG,\cT)$ does not act on $F^\times$, these isomorphisms are
$\cW^G$-equivariant if we endow the right hand side with the diagonal $\cW^G$-action.
Thus \eqref{eq:cT0} determines a $\cW^G$-equivariant isomorphism of character groups 
\begin{equation}\label{eq:split}
\Irr (\cT) \cong \Irr (\cT_0) \times \Irr (X_* (\cT)) = \Irr (\cT_0) \times X_{\unr}(\cT) .
\end{equation}

\begin{lem}\label{lem:cBernstein} 
Let $\chi$ be a character of $\cT$, and let  
\begin{align}\label{artin}
\fs = [\cT,\chi]_{\cG}
\end{align}
be the inertial class of the pair $(\cT,\chi)$.
Then $\fs$ determines, and is determined by, the $\cW^G$-orbit of a smooth morphism
\[
c^\fs \colon \fo_F^\times \to T.
\]
\end{lem}
\begin{proof}
There is a natural isomorphism
\[
\Irr (\cT) = \Hom (F^\times \otimes_{\Zset} X_* (\cT),\C^\times) \cong
\Hom (F^\times ,\C^\times \otimes_\Zset X^* (\cT)) = \Hom (F^\times ,T) . 
\]
Together with \eqref{eq:split} we obtain isomorphisms
\begin{align*}
& \Irr (\cT_0) \cong \Hom (\fo_F^\times ,T) , \\
& X_{\unr}(\cT) \cong \Hom (\Zset ,T) = T . 
\end{align*}
Let $\hat \chi \in \Hom (F^\times ,T)$ be the image of $\chi$ under these isomorphisms. 
By the above the restriction of $\hat \chi$ to $\fo_F^\times$ is not disturbed by 
unramified twists, so we take that as $c^\fs$. Conversely, by \eqref{eq:split} $c^\fs$ 
determines $\chi$ up to unramified twists. Two elements of $\Irr (\cT)$ are 
$\cG$-conjugate if and only if they
are $\cW^G$-conjugate so, in view of \eqref{artin}, the $\cW^G$-orbit
of the $c^\fs$ contains the same amount of information as $\fs$.
\end{proof}

We define
\begin{equation}
H := Z_G(\im \, c^\fs).
\end{equation}
The following crucial result is due to Roche, see \cite[p. 394 -- 395]{Roc}.     
  
\begin{lem} \label{lem:Roche}  
The group  $H^\fs$ is connected, and the finite group $W^\fs$ is the Weyl group of $H^\fs$:
\[
W^\fs = \cW^{H^\fs}
\]
\end{lem}

\section{Comparison of different parameters}
\label{sec:Borel}

\subsection{Varieties of Borel subgroups} 
We clarify some issues with different varieties of Borel subgroups and different
kinds of parameters arising from them.

Let $\mathbf{W}_F$ denote the Weil group of $F$, let $\mathbf{I}_F$ be the inertia
subgroup of $\mathbf{W}_F$. 
Let $\mathbf{W}_F^{\der}$ denote the closure of the commutator subgroup of 
$\mathbf{W}_F$, and write $\mathbf{W}_F^{\ab} = \mathbf{W}_F/\mathbf{W}^{\der}_F$. 
The group of units in $\mathfrak{o}_F$ will be denoted $\fo_F^\times$.

Next, we consider conjugacy classes in $G$ of continuous morphisms
\[
\Phi\colon \mathbf{W}_F\times \SL_2 (\Cset) \to G
\] 
which are rational on $\SL_2 (\Cset)$ and such that $\Phi(\mathbf{W}_F)$ 
consists of semisimple elements in $G$. 

Let $B_2$ be the upper triangular Borel subgroup in $\SL_2 (\Cset)$.
Let $\mathcal B^{\Phi (\mathbf{W}_F \times B_2)}$ denote the variety of Borel 
subgroups of $G$ containing $\Phi(\mathbf{W}_F \times B_2)$.
The variety $\mathcal B^{\Phi (\mathbf{W}_F \times B_2)}$ is non-empty if and 
only if $\Phi$ factors through $\mathbf W_F^{\ab}$, see \cite[\S 4.2]{Reed}.  
In that case, we view the domain of $\Phi$ to be $F^{\times} \times \SL_2 (\Cset)$: 
\[
\Phi\colon F^{\times} \times \SL_2 (\Cset) \to G.
\]
In Section \ref{subsec:enp} we show
how such a Langlands parameter $\Phi$ can be enhanced with a parameter $\rho$.

We start with the following data: a point $\fs = [\cT, \chi]_{\cG}$ and an $L$-parameter
\[
\Phi \colon F^{\times} \times \SL_2 (\C) \to G 
\]
for which
\[
\Phi |\fo^{\times}_F = c^\fs.
\]
This data creates the following items:
\begin{equation}\label{H}
\begin{aligned}
& t: = \Phi(\varpi_F, I),\\
& x := \Phi \left( 1, \matje{1}{1}{0}{1} \right) ,\\
& M: = \Z_H (t) .
\end{aligned}
\end{equation}
We note that $\Phi (\fo^{\times}_F) \subset \Z(H)$ and that $t$ commutes with 
$\Phi (\SL_2 (\C)) \subset M$. 

For $\alpha \in \Cset^{\times}$ we define the following matrix in $\SL_2 (\Cset)$:
\[
Y_{\alpha} = \matje{\alpha}{0}{0}{\alpha^{-1}} .
\]
For any $q^{1/2} \in \C^\times$ the element 
\begin{equation}\label{eq:S.12} 
t_q := t \Phi \big( Y_{q^{1/2}} \big) 
\end{equation} 
satisfies the familiar relation $t_q x t_q^{-1} = x^q$. Indeed 
\begin{equation}\label{eq:tqx} \begin{split}
t_q x t_q^{-1} & = t \Phi (Y_{q^{1/2}}) \Phi \matje{1}{1}{0}{1} 
\Phi (Y_{q^{1/2}}^{-1}) t^{-1} \\
& = t \Phi \big( Y_{q^{1/2}} \matje{1}{1}{0}{1} Y_{q^{1/2}}^{-1} \big) t^{-1} \\
& = t \Phi \matje{1}{q}{0}{1} t^{-1} = x^q . 
\end{split} \end{equation}
Notice that $\Phi (\fo^{\times}_F)$ lies in every Borel subgroup of $H$, because it
is contained in $\Z(H)$. We abbreviate $\Z_H (\Phi) = 
\Z_H (\im \Phi)$ and similarly for other groups.

\begin{lem}\label{inc}
The inclusion map 
$\Z_H (\Phi)  \to  \Z_H (t,x)$
is a  homotopy equivalence. 
\end{lem}
\begin{proof} 
Our proof depends on \cite[Prop. 3.7.23]{CG}. There is a Levi decomposition
\[
\Z_{H} (x) = \Z_{H} (\Phi (\SL_2 (\C))) U_x 
\]
where $\Z_{H} (\Phi (\SL_2 (\C))$ a maximal reductive subgroup of $\Z_H(x)$ and 
$U_x$ is the unipotent radical of $\Z_H(x)$. Therefore
\begin{equation}\label{eq:S.1}
\Z_{H} (t,x) = \Z_{H} (\Phi) \Z_{U_x}(t)  
\end{equation}
We note that $\Z_{U_x}(t) \subset U_x$ is contractible, because it is a unipotent complex group. 
It follows that
\begin{equation}\label{eq:S.10}
\Z_{H} (\Phi) \to \Z_{H} (t,x) 
\end{equation}
is a homotopy equivalence.
\end{proof}

If a group $A$ acts on a variety $X$, let $\cR(A, X)$ denote the set of irreducible
representations of $A$ appearing in the homology $H_*(X)$.

The variety of Borel subgroups of $G$ which contain $\Phi(\mathbf{W}_F \times B_2)$ 
will be denoted $\cB_G^{\Phi(\mathbf{W}_F \times B_2)}$ and the 
variety of Borel subgroups of $H$ containing $\{t,x\}$ will be denoted $\cB^{t,x}_H$.  

Lemma~\ref{inc} allows us to define
\[
A: = \pi_0( \Z_H (\Phi))  = \pi_0( \Z_H (t,x)).
\]

\begin{thm}\label{Rgroup}  
We have
\[
\cR(A, \cB^{\Phi(\mathbf{W}_F \times B_2)}) = \cR(A, \cB^{t,x}_H).
\]
\end{thm}
\begin{proof} 
This statement is equivalent to \cite[Lemma 4.4.1]{Reed} with a minor 
adjustment in his proof.   To translate into Reeder's paper, write
\[
t_q = \tau, Y_q = \tau_u, x = u, t = s.
\]
The adjustment consists in the observation that the Borel subgroup $B$ of $H$ 
contains $\{x,t_q,Y_q\}$ if and only if $B$ contains
$\{x,t,Y_q\}$.   This is because $t = t_qY_q^{-1}$.   
Therefore, in the conclusion of his proof, $\cB^{\tau, u}_H$, which is $\cB_H^{t_q, x}$, 
can be replaced by $\cB_H^{t,x}$.   
\end{proof}

In the following sections we will make use of two different but related 
kinds of parameters. 
\vspace{2mm}

\subsection{Enhanced Langlands parameters} 
\label{subsec:enp}   
Let $\bW_F$ denote the Weil group of $F$.  Via the Artin reciprocity map a Langlands 
parameter $\Phi$ for the principal 
series of $\cG$ will factor through $F^{\times} \times \SL_2(\C)$:
\begin{align}\label{Phi}
\Phi : \bW_F \times \SL_2(\C) \to F^{\times} \times \SL_2(\C) \to G.
\end{align}
Such a parameter is \emph{enhanced} in the following way.   Let $\rho$ be an irreducible 
representation of the component group of the centralizer of the image of $\Phi$:
\[
\rho \in \Irr \, \pi_0 ( Z_G(\im \, \Phi)).
\]
The pair $(\Phi, \rho)$ will be called an \emph{enhanced Langlands parameter}.   
   
We rely on Reeder's
classification of the constituents of a given principal series representation of $\cG$,  
see \cite[Theorem 1, p.101-102]{Reed}.    Reeder's theorem amounts to a 
local Langlands correspondence for the principal series 
of $\cG$.    Reeder uses only enhanced Langlands parameters with a particular geometric origin, 
namely those which occur in the homology of a certain variety of Borel subgroups of $G$. 
 
Let $B_2$ denote the standard Borel subgroup of $\SL_2(\C)$.   
For a Langlands parameter as in \eqref{Phi}, the variety of Borel subgroups 
$\mathcal B_G^{\Phi (\mathbf W_F \times B_2)}$ is nonempty, and the centralizer
$\Z_G (\Phi)$ of the image of $\Phi$ acts on it. Hence the group of components
$\pi_0 (\Z_G (\Phi))$ acts on the homology $H_* \big( \mathcal B_G^{\Phi (\mathbf W_F 
\times B_2)} ,\C \big)$. We call an irreducible representation $\rho$ of 
$\pi_0 (\Z_G (\Phi))$ \emph{geometric} if
\[
\rho\in\cR\left(\pi_0 (\Z_G (\Phi)),\mathcal B_G^{\Phi (\mathbf W_F \times B_2)}\right).
\] 
Consider the set of enhanced Langlands parameters $(\Phi,\rho)$ for which $\rho$ is geometric.
The group $G$ acts on these parameters by 
\begin{equation}\label{eq:defKLRparameter}
g \cdot (\Phi,\rho) = (g \Phi g^{-1}, \rho \circ \mathrm{Ad}_g^{-1}) 
\end{equation}
and we denote the corresponding equivalence class by $[\Phi,\rho ]_G$.  

\begin{defn}\label{Psi} 
Let $\Psi(G)_{\en}^\fs$ denote the set of $H$-conjugacy classes of enhanced 
parameters $(\Phi, \rho)$ for $\cG$ such that 
\begin{itemize}
\item $\rho$ is geometric;
\item $\Phi | \fo^{\times} = c^\fs$.
\end{itemize}
\end{defn}

Let us define a topology on $\Psi(G)_{\en}^\fs$.
For any $(\Phi, \rho) \in \Psi(G)_{\en}^\fs$ the element $x = \Phi ( \matje{1}{1}{0}{1} ) \in H$
is unipotent and $t = \Phi (\varpi_F,I) \in H$ is semisimple. By the Jacobson--Morozov Theorem 
the $H$-conjugacy class of $\Phi$ is determined completely by the $H$-conjugacy class of 
$(t,x)$, see \cite[\S 4.2]{Reed}. 
We endow the finite set $\mathfrak U^\fs$ of unipotent conjugacy classes in $H$ with the discrete 
topology and we regard the space of semisimple conjugacy classes in $H$ as the algebraic variety 
$T^\fs / W^\fs$. On $T^\fs / W^\fs \times \mathfrak U^\fs$ we take the product topology and we
endow $\Psi(G)_{\en}^\fs$ with the pullback topology from $T^\fs / W^\fs \times \mathfrak U^\fs$,
with respect to the map $(\Phi,\rho) \mapsto (t,x)$.

Notice that for this topology the $\rho$ does not play a role, two elements 
of $\Psi (G)^\fs_{\en}$ with the same $\Phi$ are inseparable.

\begin{thm}\cite{Reed}\label{Reed}   
Suppose that the residual characteristic of $F$ satisfies Condition \ref{CC}.   
\begin{enumerate}
\item There is a canonical continuous bijection
\[
\Psi(G)_{\en}^\fs \to \Irr (\cG)^\fs  .
\]
\item This bijection maps the set of enhanced Langlands parameters $(\Phi,\rho)$ for which 
$\Phi (F^\times)$ is bounded onto $\Irr (\cG)^\fs \cap \Irr (\cG)_{\mathrm{temp}}$.
\item If $\sigma \in \Irr(\cG)^\fs$ corresponds to $(\Phi,\rho)$, then the cuspidal support 
$\pi^\fs (\sigma) \in T^\fs / W^\fs$, considered as a semisimple conjugacy class in $H^\fs$, 
equals $\Phi \big( \varpi_F, Y_{q^{1/2}} \big)$.
\end{enumerate}
\end{thm}
\begin{proof} (1) The canonical bijection is Reeder's classification of the constituents 
of a given principal series representation, see \cite[Theorem 1, p.101 -- 102]{Reed}. 

First he associates to
$\Phi$ a finite length "standard" representation of $\cG$, say $M_{t,x}$, with a unique maximal 
semisimple quotient $V_{t,x}$. Then $(\Phi,\rho)$ is mapped to an irreducible constituent of 
$V_{t,x}$. To check that the bijection is continuous with respect to the above topology, it 
suffices to see that $M_{t,x}$ depends continuously on $t$ when $x$ is fixed. This property
is clear from \cite[\S 3.5]{Reed}.\\
(2) Reeder's work is based on that of
Kazhdan--Lusztig, and it is known from \cite[\S 8]{KL} that the tempered $\cG$-representations 
correspond precisely to the set of bounded enhanced L-parameters in the setting of \cite{KL}. 
As the constructions in \cite{Reed} preserve temperedness, this characterization remains valid
in Reeder's setting.\\
(3) The element $\Phi \big( \varpi_F, Y_{q^{1/2}} \big) \in H$ is the same as $t_q$
in \eqref{eq:S.12}, up to $H$-conjugacy. In the setting of Kazhdan--Lusztig,
it is known from \cite[5.12 and Theorem 7.12]{KL} that property (3) holds.  As in (2), this
is respected by the constructions of Reeder that lead to (1).
\end{proof}

\subsection{Affine Springer parameters}
\label{par:affSpringer}
As before, suppose that $t \in H$ is semisimple and that $x \in \Z_H (t)$ is unipotent. 
Then $\Z_H (t,x)$ acts on $\mathcal B_H^{t,x}$ and $\pi_0 (\Z_H (t,x))$ acts on the
homology of this variety. In this setting we say that $\rho_1 \in \Irr \big( 
\pi_0 (\Z_H (t,x)) \big)$ is \emph{geometric} if it belongs to 
$\cR\left(\pi_0 (\Z_H (t,x)),\mathcal B_H^{t,x} \right)$.

For the affine Springer parameters it does not matter whether we 
consider the total homology or only the homology in top degree. Indeed, it follows 
from \cite[bottom of page~296 and Remark 6.5]{Shoji} that any irreducible representation 
$\rho_1$ which appears in 
$H_* \big( \mathcal B_H^{t,x} ,\C \big)$, already appears in the top homology of this 
variety. Therefore, we may refine Theorem~\ref{Rgroup} as follows:

\begin{thm}\label{Rgroup_refined} 
\[
\cR(A, \cB^{\Phi(\mathbf{W}_F \times B_2)}) = \cR^{\top}(A, \cB^{t,x}_H),
\]
where $\top$ refers to highest degree in which the homology is nonzero, 
the real dimension of $\mathcal B_H^{t,x}$.
\end{thm}

We call such triples $(t,x,\rho_1)$ affine Springer parameters for $H$,
because they appear naturally in the representation theory of the affine Weyl group
associated to $H$. The group $H$ acts on such parameters by conjugation, and we 
denote the conjugacy classes by $[t,x,\rho_1]_H$.   

\begin{defn}
The set of $H$-conjugacy classes of affine Springer parameters will be denoted $\Psi(H)_{\aff}$.
\end{defn}  

Notice that the projection on the first coordinate is a canonical map
$\Psi (H)_\aff \to T / \cW^H$. We endow $\Psi(H)_{\aff}$ with a topology in the same way
as we did for $\Psi(G)_{\en}^\fs$, as the pullback of the product topology on
$T / \cW^H \times \mathfrak U^\fs$ via the map $[t,x,\rho_1]_H \mapsto (t,x)$.

For use in Theorem \ref{thm:bijection} we recall the parametrization of 
irreducible representations of $X^* (T) \rtimes \mathcal W^H$ from \cite{Kat}.
Let $t \in T$ and let $x \in M^\circ = \Z_H (t)^\circ$ be unipotent.
Kato defines an action of $X^* (T) \rtimes \mathcal W^H$ on the top homology 
$H_{d(x)}(\mathcal B^{t,x}_H,\C)$, which commutes with the action of 
$\Z_H (t,x)$ induced by conjugation of Borel subgroups.
By \cite[Proposition 6.2]{Kat} there is an isomorphism of 
$X^* (T) \rtimes \mathcal W^H$-representations
\begin{equation}\label{eq:indKato}
H_{d(x)}(\mathcal B^{t,x}_H,\C) \cong
\mathrm{ind}_{X^* (T) \rtimes \cW^{M^\circ}}^{X^* (T) \rtimes \cW^H}
\big( \C_t \otimes H_{d(x)}(\mathcal B^{x}_{M^\circ},\C) \big) .
\end{equation}
Here $H_{d(x)}(\mathcal B^{x}_{M^\circ},\C)$ is a representation occurring in
the Springer correspondence for $\cW^{M_0}$, promoted to a representation of
$X^* (T) \rtimes \cW^H$ by letting $X^* (T)$ act trivially. Hence \eqref{eq:indKato} 
has central character $\cW^H t$. We note that the underlying vector space of this 
representation does not depend on $t$, and that this determines an algebraic family of 
$X^* (T) \rtimes \mathcal W^H$-representations parametrized by $T^{\cW^{M_0}}$.
Let $\rho_1 \in \Irr \big( \pi_0 (\Z_H (t,x)) \big)$. By \cite[Theorem 4.1]{Kat} the 
$X^* (T) \rtimes \mathcal W^H$-representation
\begin{equation}\label{eq:KatoMod}
\mathrm{Hom}_{\pi_0 (\Z_H (t,x))} \big( \rho_1, H_{d(x)}(\mathcal B^{t,x}_H,\C) \big)  
\end{equation}
is either irreducible or zero.  Moreover every irreducible representation of 
$X^* (T) \rtimes \mathcal W^H$ is obtained in this way, and the data $(t,x,\rho_1)$
are unique up to $H$-conjugacy. So Kato's results provide a natural bijection
\begin{equation}\label{eq:affSpringer}
\Psi (H)_\aff \to \Irr (X^* (T) \rtimes \cW^H ) .
\end{equation}
This generalizes the Springer correspondence for finite Weyl groups, which can be 
recovered by considering the representations on which $X^* (T)$ acts trivially.

In \cite{KL,Reed} there are some indications that the above kinds of parameters are
essentially equivalent. The next result allows us to make this precise
in the necessary generality.

\begin{thm}\label{compareParameters}
Let $\fs$ be a Bernstein component in the principal series, associate 
$c^\fs \colon \fo_F^\times \to T$ 
to it as in Lemma \ref{lem:cBernstein} and let $H$ be as in (\ref{H}).   
There are natural bijections between $H$-equivalence classes of:
\begin{itemize}
\item enhanced Langlands parameters $(\Phi, \rho)$ for $\cG$,  
with $\rho$ geometric and $\Phi \big|_{\fo_F^\times} = c^\fs$;
\item affine Springer parameters for $H$.
\end{itemize}
In other words we have a homeomorphism
\[
\Psi(G)_{\en}^\fs \simeq \Psi(H)_{\aff}.
\]
\end{thm}
\begin{proof}   
An $L$-parameter gives rise to the ingredients $t,x$  in an affine Springer parameter 
in the following way. For an $L$-parameter
\[
\Phi \colon F^{\times} \times \SL_2(\C) \to G
\]
we set $t = \Phi(\varpi_F, 1)$ and $x = \Phi \big( 1, \matje{1}{1}{0}{1} \big)$.

Conversely, we work with the Jacobson--Morozov Theorem \cite[p. 183]{CG}.  
Let $x$ be a unipotent element in $M^0$. There exist rational homomorphisms 
\begin{equation} \label{eqn:gamt}
\gamma \colon \SL_2 (\Cset) \to M^0 \quad \text{with} \quad 
\gamma \big( \matje{1}{1}{0}{1} \big) = x ,
\end{equation}
see \cite[\S 3.7.4]{CG}. Any two such homomorphisms  $\gamma$ are conjugate by
elements of $\Z_{M^\circ}(x)$. 
Define the Langlands parameter $\Phi$ as follows:
\begin{equation}\label{eqn:Phi}
\Phi \colon F^{\times} \times \SL_2 (\Cset) \to G, \qquad  
(u\varpi_F^n,Y) \mapsto c^\fs (u) \cdot t^n\cdot \gamma(Y) 
\end{equation}
for all  $u \in \fo_F^\times, \; n \in \Zset,\; Y \in \SL_2 (\Cset)$.

Note that the definition of $\Phi$ uses the appropriate data: 
the semisimple element $t \in T$, the map $c^\fs$, and the 
homomorphism $\gamma$ (which depends on  $x$).  

Since $x$ determines $\gamma$ up to $M^\circ$-conjugation, $c^\fs,x$ and $t$ 
determine $\Phi$ up to conjugation by their common centralizer in $G$. 
Notice also that one can recover $c^\fs, x$ and $t$ from $\Phi$ and that
\begin{equation}\label{eq:hPhi}
h (\alpha): = \Phi (1, Y_\alpha) 
\end{equation}
defines a cocharacter $\C^\times \to T$.   


To complete $\Phi$ or $(t,x)$ to a parameter of the appropriate kind,
we must add an irreducible representation $\rho$ or $\rho_1$.
Then the bijectivity follows from Theorem~\ref{Rgroup_refined}.

It is clear that the above correspondence between $\Phi$ and $(t,x)$ is 
continuous in both directions. In view of the chosen topologies on
$\Psi(G)_{\en}^\fs$ and $\Psi(H^\fs)_{\aff}$, this implies that the
bijection is a homeomorphism. 
\end{proof}

\section{Structure theorem}
\label{sec:unip}

Let $\fs \in \mathfrak B (\cG,\cT)$ and construct $c^\fs$ as in 
Lemma \ref{lem:cBernstein}. We note that the set of enhanced Langlands parameters 
$\Phi(G)_\en^\fs$ is naturally labelled by the unipotent classes in $H$: 
\begin{equation}
\Phi(G)_\en^{\fs,[x]} := \big\{ (\Phi,\rho) \in \Phi(G)_\en^\fs \mid 
\Phi \big( 1, \matje{1}{1}{0}{1} \big) \text{ is conjugate to } x \big\} .
\end{equation}
Via Theorem \ref{compareParameters} and \eqref{eq:affSpringer} the sets $\Phi(G)_\en^\fs$ 
and $\Irr (X^* (T) \rtimes \cW^H)$ are naturally in 
bijection with $\Psi (H)_\aff$. In this way we can associate to any of these
parameters a unique unipotent class in $H$:
\begin{equation} \label{eq:labelling}
\begin{aligned}
& \Irr (\cG )^\fs && = \; \bigcup\nolimits_{[x]} \Irr (\cG )^{\fs,[x]} , \\
& \Psi (H)_\aff && = \; \bigcup\nolimits_{[x]} \Psi (H)_\aff^{[x]} ,\\
& \Irr (X^* (T) \rtimes \cW^H) && 
= \; \bigcup\nolimits_{[x]} \Irr (X^* (T) \rtimes \cW^H)^{[x]} .
\end{aligned}
\end{equation}
As $\Irr (\cG )^\fs = \Irr (\cH^\fs)$ and $\Irr (X^* (T) \rtimes \cW^H) =
\Irr (\C [X^* (T) \rtimes \cW^H])$, these spaces are endowed with the Jacobson topology
from the respective algebras $\cH^\fs$ and $\C [X^* (T) \rtimes \cW^H]$.

Recall from Section \ref{sec:extquot} that 
\[
\widetilde{T^\fs} = \{ (w,t) \in W^\fs \times T^\fs \mid w t = t \}
\]
and $T^\fs /\!/ W^\fs = \widetilde{T^\fs} / W^\fs$. We endow $\widetilde{T^\fs}$ with
the product of the Zariski topology on $T^\fs \cong T$ and the discrete topology on 
$W^\fs = \cW^H$. Then $T^\fs /\!/ W^\fs$ with the quotient topology from $\widetilde{T^\fs}$ 
becomes a disjoint union of algebraic varieties. The following result enables us to transfer 
the labellings \eqref{eq:labelling} to $T^\fs \q W^\fs$.

\begin{thm} \label{thm:bijection}
There exists a bijection 
$\tilde{\mu}^{\fs} : T^\fs \q W^\fs \to \Irr (X^* (T) \rtimes W^\fs)$ such that: 
\begin{itemize}
\item[(1)] $\tilde{\mu}^{\fs}$ respects the projections to $T^\fs / W^\fs$;
\item[(2)] for every unipotent class $x$ of $H$, the inverse image 
$(\tilde{\mu}^\fs )^{-1} \Irr (X^* (T) \rtimes W^\fs)^{[x]}$ is a union of connected 
components of $T^\fs \q W^\fs$. 
\end{itemize}
\end{thm}
\begin{proof}
Consider the ring $R (X^* (T^\fs) \rtimes W^\fs)$ of virtual finite dimensional complex 
representations $X^* (T^\fs) \rtimes W^\fs$. Its canonical $\mathbb Z$-basis is
\[
\Irr (X^* (T) \rtimes W^\fs) = \{ \tau (t,x,\rho_1) : (t,x,\rho_1) \in \Psi (H)_\aff \} .
\]
The $\Q$-vector space
\[
R_\Q (X^* (T) \rtimes W^\fs) := \Q \otimes_{\mathbb Z} R (X^* (T) \rtimes W^\fs)
\]
possesses another useful basis coming from $T^\fs \q W^\fs$. Given $w \in W^\fs$,
let $C_w$ be the cyclic subgroup it generates. We define a character $\chi_w$ of $C_w$
by the formula 
\[
\chi_w (w^n ) = \exp (2 \pi i n / |C_w|). 
\]
For any $t \in (T^\fs)^w$ we obtain a character $\C_t \otimes \chi_w$ of 
$X^* (T^\fs) \rtimes C_w$. We induce that to a $X^* (T^\fs) \rtimes W^\fs$-representation
\[
\chi (w,t) := \mathrm{ind}^{X^* (T) \rtimes W^\fs}_{X^* (T) \rtimes C_w}
( \C_t \otimes \chi_w ) 
\]
with central character $W^\fs t$. The representation $\chi (w,t)$ is irreducible whenever
$t$ is a generic point of $(T^\fs)^w$, which in this case means simply that $t$ is not fixed
by any element of $W^\fs \setminus C_w$. It is easy to see that $\chi (w,t) \cong \chi (w',t')$
if and only if $(w,t)$ and $(w',t')$ are $W^\fs$-associate (which means that they determine
the same point of $T^\fs \q W^\fs$). Moreover, it follows from Clifford theory and Artin's 
theorem \cite[Theorem 17]{Ser} that 
\[
\{ \chi (w,t) : [w,t] \in T^\fs \q W^\fs \}
\]
is a $\Q$-basis of $R_\Q (X^* (T) \rtimes W^\fs)$, see \cite[(40)]{Sol2}.

Now we construct the desired map $\tilde{\mu}^\fs$, with a recursive procedure. Take 
$0 \leq d \leq \dim_\C (T)$.   With $w \in W^\fs$,  define
\[
(T^\fs)^w: = \{t \in T^\fs : wt = t\}.
\]  
Suppose that we already have defined $\tilde{\mu}^\fs$
on all connected components of $T^\fs \q W^\fs$ of dimension $<d$, and that
\begin{equation}\label{eq:span}
\begin{split}
\text{span } \tilde{\mu}^\fs \big( \{ [w,t] \in T^\fs \q W^\fs : 
\dim (T^\fs)^w < d \} \big) \; \cap \\
\text{span } \{ \chi (w,t) : (w,t) \in \widetilde{T^\fs}, \dim (T^\fs)^w \geq d \} \; = 0 .
\end{split}
\end{equation} 
Fix $t_1 \in T^\fs$. Since \eqref{eq:indKato} has central character $W^\fs t$, 
both $\{ \tau (t,x,\rho_1) : (t_1,x,\rho_1) \in \Psi (H)_\aff$ and 
$\{ \chi (w,t_1) : [w,t_1] \in T^\fs \q W^\fs\}$ are bases of the finite dimensional 
$\Q$-vector space $R_\Q (X^* (T) \rtimes W^\fs)_{W^\fs t_1}$ spanned by the 
$X^* (T) \rtimes W^\fs$-representations which admit the central character $W^\fs t_1$. 
From this and the assumption \eqref{eq:span} we see
that we can find, for very $w \in W^\fs$ fixing $t_1$, an irreducible constituent 
$\tilde{\mu}^\fs ([w,t_1])$ of $\chi (w,t_1)$ such that
\begin{multline*}
\text{span } \tilde{\mu}^\fs \big( \{ [w,t_1] \in T^\fs \q W^\fs : 
\dim (T^\fs)^w \leq d \} \big) \\
\cup \; \text{span } \{ \chi (w,t_1) : (w,t_1) \in \widetilde{T^\fs} ,\, \dim (T^\fs)^w > d \} 
\end{multline*}
is again a $\Q$-basis of $R_\Q (X^* (T) \rtimes W^\fs)_{W^\fs t_1}$. In this way we construct 
$\tilde{\mu}^\fs$ on the $d$-dimensional connected components of $T^\fs \q W^\fs$, such that 
\eqref{eq:span} becomes valid for $d+1$. 
Thus we obtain a bijection $\tilde{\mu}^\fs : T^\fs \q W^\fs \to \Irr (X^* (T) \rtimes W^\fs)$ 
which satisfies (1). 

It remains to check (2). Fix $w \in W^\fs$ and consider a connected component $(T^\fs)^w_i$ 
of $(T^\fs)^w$. For generic $t \in (T^\fs)^w_i ,\; \chi (w,t) = \tilde{\mu}^\fs ([w,t])$
is irreducible. We note that both $\{ \chi (w,t) : t \in (T^\fs)^w_i \}$ and \eqref{eq:indKato} 
(with fixed $x$) are algebraic families of $X^* (T) \rtimes W^\fs$-representations parametrized 
by $(T^\fs)^w_i$. That set is an irreducible algebraic variety because it is a coset of the 
neutral component of $(T^\fs)^w$, which is a subtorus of $T^\fs$. It follows that the 
irreducible $\chi (w,t)$ are all contained in $\Irr (X^* (T) \rtimes W^\fs)^{[x]}$ for one $x$. 
By continuity $\chi (w,t)$ is a subrepresentation of $H_{d(x)} (\mathcal B_H^{t,x},\C)$ for 
all $t \in (T^\fs)^w_i$, which implies that the subquotient $\tilde{\mu}^\fs ([w,t])$ of 
$\chi (w,t)$ has the form $\tau (t,x,\rho_1)$ for the same $x$. 
Hence $\tilde{\mu}^\fs ( [w,(T^\fs)^w_i] ) \subset \Irr (X^* (T) \rtimes W^\fs)^{[x]}$.
\end{proof}

We remark that with more effort it is possible to refine the above construction so that 
$\tilde{\mu}^\fs$ becomes continuous. But since we do not need that refinement, we refrain from 
writing it down here.

\begin{thm}\label{split} 
Let $\cG$ be a split reductive $p$-adic group with connected 
centre, such that the residual characteristic satisfies Condition \ref{CC}. 
Then, for each point $\fs$ in the principal series of $\cG$,  we have a continuous bijection
\[
\mu^\fs : T^\fs \q W^\fs \to \Irr(\cG)^\fs .
\]
It maps $T^\fs_\cpt \q W^\fs$ onto $\Irr(\cG)^\fs \cap \Irr (\cG)_{\mathrm{temp}}$.
\end{thm}
\begin{proof} 
To get the bijection $\mu^\fs$, apply Theorems \ref{compareParameters}, 
\ref{Reed}.(1) and \ref{thm:bijection}. 
The properties (1) and (2) in Theorem \ref{thm:bijection} ensure that the composed map
\[
T^\fs \q W^\fs \to \Irr (X^* (T) \rtimes W^\fs) \to \Psi (H)_\aff
\]
is continuous, so $\mu^\fs$ is continuous as well.

By Theorem \ref{thm:bijection} $T^\fs_\cpt \q W^\fs$ is first mapped bijectively 
to the set of parameters in $\Psi (H)_\aff$ with $t$ compact. From the proof of Theorem 
\ref{compareParameters} we see that the latter set is mapped onto the set of enhanced 
Langlands parameters $(\Phi,\rho)$ with $\Phi \big|_{\fo_F^\times} = c^\fs$ and 
$\Phi (\varpi_F)$ compact. These are just the bounded enhanced Langlands parameters, so by 
Theorem \ref{Reed}.(2) they correspond to $\Irr(\cG)^\fs \cap \Irr (\cG)_{\mathrm{temp}}$.
\end{proof}

\section{Correcting cocharacters and L-packets}
\label{sec:cochar}

In this section we construct "correcting cocharacters" on the extended quotient
$T^\fs /\!/ W^\fs$. These measure the difference between the canonical projection
$T^\fs /\!/ W^\fs \to T^\fs / W^\fs$ and the composition of $\mu^\fs$ (from Theorem
\ref{split}) with the cuspidal support map $\Irr (\cG)^\fs \to T^\fs / W^\fs$.
As conjectured in \cite{ABPS1}, they show how to determine when two elements of 
$T^\fs \q W^\fs$ give rise to $\cG$-representations in the same L-packet.

Every enhanced Langlands parameter $(\Phi,\rho)$ naturally determines a cocharacter $h_\Phi$ 
and elements $\theta (\Phi,\rho,z) \in T^\fs$ by
\begin{equation}\label{eq:defhPhi}
\begin{aligned}
& h_\Phi (z) = \Phi \big( 1,\matje{z}{0}{0}{z^{-1}} \big) ,\\
& \theta (\Phi,\rho,z) = \Phi \big( \varpi_F, \matje{z}{0}{0}{z^{-1}} \big) =
\Phi (\varpi_F) h_\Phi (z) .
\end{aligned}
\end{equation}
Although these formulas obviously do not depend on $\rho$, it turns out to be convenient to 
include it in the notation anyway.
However, in this generality we would end up with infinitely many correcting cocharacters, 
most of them with range outside $T$. To reduce to finitely many cocharacters with values 
in $T$, we must fix some representatives for $\mathfrak U^\fs$ in $H$.

Fix a Borel subgroup $B_H$ of $H$ containing $T$. Following the recipe from the
Bala--Carter classification \cite[Theorem 5.9.6]{Car} we choose a set of
representatives $\fU^\fs \subset B_H$ for the unipotent classes of $H$.

\begin{lem}\label{lem:Bala-Carter}
Every commuting pair $(t,x)$ with $t \in H$ semisimple and $x \in H$ 
unipotent is conjugate to one with $x \in \fU^\fs$ and $t \in T$.
\end{lem}
\begin{proof}
Obviously we can achieve that $x \in \fU^\fs$ via conjugation in $H$. Choose a homomorphism 
of algebraic groups $\gamma : \SL_2 (\C) \to H$ with $\gamma ( \matje{1}{1}{0}{1} ) = x$.
As noted in \eqref{eqn:gamt}, such a $\gamma$ exists and is unique up to conjugation
by $Z_H (x)$. The constructions for the Bala--Carter theorem in \cite[\S 5.9]{Car} entail
that we can choose $\gamma$ such that $\gamma (Y_\alpha) \in T$ for all 
$\alpha \in \C^\times$. On the other hand, we can also construct such a $\gamma$ inside
the reductive group $Z_H (t)$. So, upon conjugating $t$ by a suitable element of
$Z_H (x)$, we can achieve that the standard maximal torus $T_x$ of $\gamma (\SL_2 (\C))$ 
is contained in $T$ and commutes with $t$. Let $S \subset H$ be a maximal torus 
containing $T_x$ and $t$. Then
\[
T = (T \cap Z_G ( \mathrm{im} \, \gamma ))^\circ Z(G) T_x ,
\]
and similarly for $S$. It follows that
\[
T \cap Z_G ( \mathrm{im} \, \gamma )^\circ \quad \text{and} \quad 
S \cap Z_G ( \mathrm{im} \, \gamma )^\circ
\]
are maximal tori of $Z_G ( \mathrm{im} \gamma )^\circ$. They are conjugate, which shows
that we can conjugate $t$ to an element of $T$ without changing $x \in \fU^\fs$.
\end{proof}

Recall that \eqref{eq:labelling} and Theorem \ref{thm:bijection} 
determine a labelling of the connected 
components of $T^\fs /\!/ W^\fs$ by unipotent classes in $H$. This enables us to define the 
correcting cocharacters: for a connected component $\mathbf c$ of $T^\fs /\!/ W^\fs$ with 
label (represented by) $x \in \mathfrak U^\fs$ let $\gamma_x = \gamma$
be as in \eqref{eqn:gamt} and $\Phi$ as in \eqref{eqn:Phi}. We take the cocharacter 
\begin{equation}\label{eq:defhx}
h_{\mathbf c} = h_x : \C^\times \to T ,\quad h_x (z) = \gamma_x \matje{z}{0}{0}{z^{-1}} .
\end{equation}
Let $\widetilde{\mathbf c}$ be a connected component of $\widetilde{T^\fs}$ that projects
onto $\mathbf c$ and centralizes $x$. In view of Lemma \ref{lem:Bala-Carter} this can always 
be achieved by adjusting by an element of $W^\fs$. We define
\begin{equation} \label{eq:defThetaz}
\begin{aligned}
& \widetilde{\theta_z} : \widetilde{\mathbf c} \to T^\fs ,& & (w,t) \mapsto 
t \, h_{\mathbf c}(z)  , \\
& \theta_z : \mathbf c \to T^\fs / W^\fs ,& & [w,t] \mapsto W^\fs t \, h_{\mathbf c}(z)  . 
\end{aligned}
\end{equation}

\begin{thm}\label{Lpackets}
Let $[w,t],[w',t'] \in T^\fs /\!/W^\fs$. Then $\mu^\fs [w,t]$ and $\mu^\fs [w',t']$ are
in the same L-packet if and only if
\begin{itemize}
\item $[w,t]$ and $[w',t']$ are labelled by the same unipotent class in $H$;
\item $\theta_z [w,t] = \theta_z [w',t']$ for all $z \in \C^\times$.
\end{itemize}
\end{thm}
\begin{proof}
Suppose that the two $\cG$-representations $\mu^\fs [w,t] = \pi (\Phi,\rho)$ and \\
$\mu^\fs [w',t'] = \pi (\Phi',\rho')$ belong to the 
same L-packet. By definition this means that $\Phi$ and $\Phi'$ are $G$-conjugate. 
Hence they are labelled by the same unipotent class, say $[x]$ with $x \in \mathfrak U^\fs$. 
By choosing suitable representatives we may assume that $\Phi = \Phi'$ and that 
$\{(\Phi,\rho),(\Phi,\rho')\} \subset \Phi (G)_\en^{\fs,[x]}$. Then
\[
\theta (\Phi,\rho,z) = \theta (\Phi,\rho',z) \text{ for all } z \in \C^\times.
\]
Although in general $\theta (\Phi,\rho,z) \neq \widetilde{\theta_z} (w,t)$, they differ only 
by an element of $W^\fs$. Hence $\theta_z [w,t] = \theta_z [w',t']$ for all $z \in \C^\times$.

Conversely, suppose that $[w,t],[w',t']$ fulfill the two conditions of the lemma. Let 
$x \in \mathfrak U^\fs$ be the representative for the unipotent class which labels them.
From Lemma \ref{lem:Bala-Carter} we see that there are
representatives for $[w,t]$ and $[w',t']$ such that $t (T^w)^\circ$ and $t' (T^{w'})^\circ$
centralize $x$. Then 
\[
\widetilde{\theta_z} (w,t) = t \, h_x (z) \quad \text{and} \quad
\widetilde{\theta_z} (w',t') = t' \, h_x (z)
\]
are $W^\fs$ conjugate for all $z \in \C^\times$. As these points depend continuously on $z$
and $W^\fs$ is finite, this implies that there exists a $v \in W^\fs$ such that
\[
v (t \, h_x (z)) = t' \, h_x (z) \quad \text{for all } z \in \C^\times .
\]
For $z = 1$ we obtain $v(t) = t'$, so $v$ fixes $h_x (z)$ for all $z$. 

Consider the minimal parabolic root subsystem $R_P$ of $R (G,T)$ that supports $h_x$. In
other words, the unique set of roots $P$ such that $h_x$ lies in a facet of type
$P$ in the chamber decomposition of $X^* (T) \otimes_\Z \R$. We write 
\[
T^P = \{ t \in T \mid \alpha (t) = 1 \; \forall \alpha \in P \}^\circ . 
\]
Then $t (T^w)^\circ$ and $t' (T^{w'})^\circ$ are subsets of $T^P$ and $v$ stabilizes $T^P$. 
It follows from \cite[Proposition B.4]{Opd} that $h_x (q^{1/2}) t T^P$ and
$h_x (q^{1/2}) t' T^P$ are residual cosets in the sense of Opdam. By the above, these two 
residual cosets are conjugate via $v \in W^\fs$. Now \cite[Corollary B.5]{Opd} says that
the pairs $(h_x (q^{1/2}) t,x)$ and $(h_x (q^{1/2}) t',x)$ are $H$-conjugate. Hence the
associated Langlands parameters are conjugate, which means that $\mu^\fs [w,t]$ and 
$\mu^\fs [w',t']$ are in the same L-packet.
\end{proof}

\begin{cor}\label{cor:properties}
Properties 1--5 from \cite[\S 15]{ABPS1} hold for $\mu^\fs$ as in 
Theorem \ref{split}, with the morphism $\theta_z$ from \eqref{eq:defThetaz}
and the labelling by unipotent classes in $H^\fs$ from \eqref{eq:labelling} 
and Theorem \ref{thm:bijection}.

Together with Theorem \ref{split} this proves the conjecture from \cite{ABPS1} 
for all Bernstein components in the principal series of a split reductive 
$p$-adic group with connected centre, such that the residual characteristic satisfies 
Condition \ref{CC}.
\end{cor}
\begin{proof}
Property (1) was shown in Theorem \ref{split}.
By the definition of $\theta_z$ \eqref{eq:defThetaz}, property (4) holds.
Property (3) is a consequence of property (4), in combination with Theorems \ref{Reed}.(3), 
\ref{split} and \ref{thm:bijection}. Property (2) follows from Theorem \ref{split} and
property (3). Property (5) is none other than Theorem \ref{Lpackets}.
\end{proof}


\begin{thebibliography}{99}
\bibitem[ABP1]{ABP1} A.-M. Aubert, P. Baum, R.J. Plymen, The Hecke algebra of a
reductive $p$-adic group: a geometric conjecture, pp. 1--34 in:
\emph{Noncommutative geometry and number theory, Eds: C. Consani and M.
Marcolli}, Aspects of Mathematics {\bf E37}, Vieweg Verlag (2006).
\bibitem[ABP2]{ABP2} A.-M. Aubert, P. Baum, R.J. Plymen,  Geometric structure in
the principal series of the $p$-adic group $G_2$,  
Represent. Theory {\bf 15} (2011), 126--169. 
\bibitem[ABPS]{ABPS1} A.-M. Aubert, P. Baum, R.J. Plymen, M. Solleveld,  
Geometric structure in smooth dual and local Langlands conjecture,
Japanese J. Math. {\bf 9} (2014) 
\bibitem[Car]{Car} R.W. Carter, \emph{Finite groups of Lie type.
Conjugacy classes and complex characters}, Pure and Applied Mathematics,
John Wiley \& Sons, New York NJ, 1985
\bibitem[ChGi]{CG} N. Chriss, V. Ginzburg, \emph{Representation theory and complex
geometry}, Birkh\"auser, 2000.
\bibitem[Kat]{Kat} S.-I. Kato,
A realization of irreducible representations of affine Weyl groups,
Indag. Math. \textbf{45.2} (1983), 193--201.
\bibitem[KaLu]{KL} D.~Kazhdan, G.~Lusztig, Proof of the Deligne-Langlands
conjecture for Hecke algebras, Invent. math. {\bf 87} (1987), 153--215.
\bibitem[Kha]{K} M. Khalkhali, Basic noncommutative geometry, EMS Lecture Series 2009.   
\bibitem[Opd]{Opd} E.M. Opdam, On the spectral decompostion of affine Hecke algebras,
J. Inst. Math. Jussieu {\bf 3} (2004), 531--648.
\bibitem[Ree]{Reed} M.~Reeder, Isogenies of Hecke algebras and a Langlands correspondence
for ramified principal series representations,  Represent. Theory {\bf
6} (2002), 101--126.
\bibitem[Roc]{Roc} A. Roche, Types and Hecke algebras for principal
series representations of split reductive $p$-adic groups, Ann.
scient. \'Ec. Norm. Sup. {\bf 31} (1998), 361--413.
\bibitem[Ser]{Ser} J.-P. Serre,
\emph{Linear representations of finite groups},
Springer Verlag, 1977.
\bibitem[Sho]{Shoji} T.~Shoji, Green functions of reductive groups over a
finite field, PSPM {\bf 47} (1987), Amer. Math. Soc., 289--301.
\bibitem[Sol1]{Sol} M.~Solleveld, On the classification of irreducible
representations of affine Hecke algebras with unequal parameters,
Represent. Theory {\bf 16} (2012), 1--87. 
\bibitem[Sol2]{Sol2} M. Solleveld, Hochschild homology of affine Hecke algebras,
Journal of Algebra {\bf 384} (2013), 1--35.
\end{thebibliography}
\end{document}